\documentclass[11pt,a4paper]{amsart}
\usepackage[utf8]{inputenc}
\usepackage[finnish,english]{babel}
\usepackage{graphicx}
\usepackage{amsmath}
\usepackage{amsfonts}
\usepackage{amssymb}
\usepackage{enumerate}

\newtheorem{thm}{Theorem}[section]

\newtheorem{lemma}[thm]{Lemma}
\newtheorem{prop}[thm]{Proposition}
\theoremstyle{definition}
\newtheorem{defn}[thm]{Definition}
\theoremstyle{remark}

\newtheorem{ex}{Example}
\numberwithin{equation}{section}

\newcommand{\eps}{\varepsilon}
\newcommand{\Real}{\mathbb R}
\newcommand{\Complex}{\mathbb C}
\newcommand{\disk}{\mathbb D}

\newcommand{\B}{\mathcal{B}}

\newcommand{\h}{\mathcal{H}}
  
\newcommand{\M}{\mathcal{M}}

\newcommand{\p}{\mathcal{P}}

\newcommand{\V}{\mathcal{V}}  

\newcommand{\ol}{\overline}

\DeclareMathOperator*{\spn}{span}
\DeclareMathOperator*{\cspn}{\overline{span}}

\begin{document}
\title
{Compact real linear operators}
\author[S. Ruotsalainen]{Santtu Ruotsalainen}
\address{Aalto University \\ Institute of Mathematics \\ P.O. Box 11100 \\ FI-00076 Aalto \\ Finland}
\email{Santtu.Ruotsalainen [at] aalto.fi}
\date{\today}
\begin{abstract}
Real linear operators emerge in a range of mathematical physics applications. In this paper spectral questions of compact real linear operators are addressed. A Lomonosov-type invariant subspace theorem for antilinear compact operators is proved. Properties of the characteristic polynomial of a finite rank real linear operator are investigated. A related numerical function, defined as a normalization of the characteristic polynomial, is studied. An extension to trace-class operators is discussed. 
\end{abstract}
\maketitle
\section{Introduction}

Real linear operators arise in various applications in mathematical physics, such as in Cald\'eron's inverse problem in the plane \cite{astalapaivarinta} and in studies related with planar elasticity \cite{putinarshapiro}. 
An operator $R$ on a separable complex Hilbert space $\h$ is said to be real linear if $R(x+y)=Rx+Ry$ and $R(rx)=rRx$ for all $x,y\in\h$ and $r\in\Real$. 
This paper is concerned with compact real linear operators  and their spectral properties. 
Compactness means that images of bounded sets have compact closures. 
An invariant subspace result for compact antilinear operators is proved. 
To study the spectrum of finite rank real linear operators, properties of the characteristic polynomial are investigated 
and a related numerical function is introduced. 

Spectral studies of operators are naturally linked with invariant subspaces. 
For compact complex linear operators, the existence of a nontrivial invariant subspace was established by von Neumann in an unpublished form. 
The subsequent  generalizations in \cite{as:invariant, br:invariant, halmos:invariant, lomonosov:invariant, michaels:hildensprf} lead to Lomonosov's theorem. 
We prove an antilinear analogue stating that given a nonzero compact antilinear operator $R$ there exists a nontrivial complex linear subspace invariant under every real linear operator in the norm closure of complex polynomials in $R$. 
Taking into account that a compact real linear operator does not necessarily have a nontrivial invariant subspace, it is not clear for which class of operators the result could be extended further.  

The spectrum of a finite rank real linear operator is given by the zero set of a characteristic polynomial $p(\lambda, \ol\lambda)$ with the inclusion of the origin when $\h$ is infinite dimensional. 
The spectral problem is challenging already in the finite dimensional case studied in \cite{eirolaetal:solmeth, rlinmatan, huhtanen:rlineigprob}. 
The characteristic polynomial is a so-called real analytic polynomial so that its zero set, if nonempty, is a plane algebraic curve. 
However, it is nontrivial to show under what assumptions the spectrum is nonempty. 
For compact real linear operators, so far, it is known that the intersection of the spectrum and a straight line through the origin consists of isolated points that accumulate at most to the origin; see \cite{reallinearoperator}. 

In this paper, it is shown that the characteristic polynomial $p$ can be cast in the form 
$$
p ( \lambda , \overline\lambda ) = v^* H v
$$
where $H=[h_{ij}]_{i,j=0}^n \in \Complex^{(n+1)\times (n+1)}$ is a Hermitian matrix and $v=[\lambda^j]_{j=0}^n \in\Complex^{n+1}$. 
This factorization allows proving a number of properties of $p$. 
First, $p$ can be represented as a weighted sum of squares $\sum d_i |p_i|^2$, where the $d_i$ are real and $p_i$ complex analytic polynomials. 
Moreover, the normalization 
$$
F (\lambda) = \frac{p(\lambda,\overline\lambda)}{ \sum_{i=0}^n |\lambda|^{2i}  }
$$
is shown to be connected to the field of values of the coefficient matrix $H$. 
More precisely, its values are convex combinations of the eigenvalues of $H$ and hence a subset of the field of values of $H$. 
When transformed to infinite dimensions, a similar numerical function for  trace-class real linear  operators is studied by utilizing a generalization of the determinant.

The contents are as follows. In Section 2 some preliminary definitions are given and Shatten classes are discussed. In Section 3 an antilinear version of Lomonosov's theorem is proved. In Section 4 the characteristic polynomial of finite rank real linear operators is studied. The so-called numerical function is introduced and its properties are investigated. A related numerical function for trace-class operators is considered. 

\section{Real linear compactness}

\subsection{Preliminaries}

Let $\h$ be a complex separable Hilbert space. 
A real linear operator $R$ on $\h$ is complex linear if $Ri=iR$, or antilinear if $Ri=-iR$. 
With respect to these extremes, a real linear operator $R$ on $\h$ can be represented uniquely as 
\begin{equation}\label{CArep}
R = C + A
\end{equation}
where $C=\frac 12 (R-iRi)$ is the complex linear and $A=\frac 12 (R+iRi)$ is the antilinear part of $R$. Further, an antilinear operator $A$ can be factored as $A=B\kappa$, where $B=A\kappa$ is complex linear and $\kappa$ is a unitary conjugation, i.e., an antilinear operator $\kappa$ that is involutory, $\kappa^2=I$, and unitary (bijective isometry). The complexification of $R$ with respect to $\kappa$ is then a complex linear operator on $\h\oplus\h$ defined by
\begin{equation}\label{complexification}
R^\Complex = 
\begin{bmatrix}
C  & A\kappa \\
\kappa A & \kappa C \kappa
\end{bmatrix} . 
\end{equation}

A real linear operator $R$ is bounded if its operator norm 
$\|R\| = \sup \{ \|Rx\| : \|x\|=1\} $ 
is finite. 
\begin{defn}
The spectrum $\sigma(R)$ consists of those $\lambda\in\Complex$ for which $R-\lambda$ does not have a bounded inverse. 
\end{defn}
As usual, $\lambda$ is an eigenvalue of $R$ if there exists $x\neq 0$ such that $Rx=\lambda x$. 


Compactness is defined using  \eqref{CArep}.
\begin{defn}
A real linear operator $R=C+A$ on $\h$ is compact if its complex linear part $C$ and antilinear part $A$ are compact. 
\end{defn}
A real linear operator $R$ is of finite rank $n$ if the so-called left-rank $n=\dim (R\h^\perp)^\perp$ is finite. Or equivalently, the dimension $n$ of the smallest complex linear subspace of $\h$ containing the range of $R$ is finite. 
Obviously, a real linear operator is compact if and only if it can be approximated in operator norm with finite rank operators, or yet equivalently, if its complex linear and antilinear parts can be approximated by finite rank complex linear and antilinear operators, respectively. 

Fairly recent examples of compact real linear operators are provided by the so-called boundary Friedrichs operator  and the Friedrichs operator of a planar domain. More elaborate applications to inverse problems can be found in \cite{astalapaivarinta}. 

\begin{ex}
Consider the Hardy space on the unit circle $\mathbb T$
$$H^2=\{ f(e^{i\theta}) = \sum_{k\in\mathbb Z} \hat f(k) e^{ik\theta} \in L^2 : \hat f(k) =0 \text{ for all } k < 0 \}  
$$
with the inner product $(f,g) = \frac{1}{2\pi}\int_0^{2\pi} f(e^{i\theta}) \overline{g(e^{i\theta})} \,d\theta$. The so-called boundary Friedrichs operator $B_a$ with symbol $a = \sum_{k\in\mathbb Z} a_k e^{ik\theta} \in L^\infty (\mathbb T)$ is defined on $H^2$ as
$$ B_a f = P a \overline f ,
$$
where $P$ is the orthogonal projection from $L^2$ onto $H^2$. For the orthonormal basis $\{ e^{ik\theta} \}_{k=0}^\infty $ we have
$$ (B_a e^{ik\theta} , e^{il\theta} ) = a_{k+l} .
$$
Thus $B_a$ is an antilinear self-adjoint operator on $H^2$ and the infinite matrix with respect to $\{ e^{ik\theta} \}_{k=0}^\infty $ is a  Hankel matrix. Furthermore, it is known that $B_a$ is compact when $a\in C(\mathbb T)$. For further reference on more general Friedrichs operators, see \cite{putinarshapiro, putinarshapiro2}. 
\end{ex}

\begin{ex}
Let $\Omega$ be a planar domain. The Friedrichs operator with symbol $a\in L^{\infty}(\Omega)$ is defined on the space of analytic square-intergable functions $AL^2(\Omega)$ as
$$ F_a : AL^2(\Omega) \to AL^2(\Omega), \quad F_a f = P a \overline f ,
$$
where $P$ is the orthogonal projection from $L^2(\Omega)$ onto $AL^2(\Omega)$, called the Bergman projection. It is known that if boundary of $\Omega$ is smooth enough and if $a\in C(\overline \Omega)$, then $F_a$ is compact. 
Here $\overline \Omega$ denotes the closure of $\Omega$. 
Again, for a more accurate description see \cite{putinarshapiro, putinarshapiro2}.  
\end{ex}

Recall that a compact complex linear operator is in the Shatten $p$-class if its singular numbers belong to  $l^p$. The singular numbers of a compact antilinear  operator $A$ are defined identically as $\sqrt{s_j(A^*A)}$, and if they belong to $l^p$, $A$ is in the Schatten $p$-class. Then we set the following for compact real linear operators.  

\begin{defn} 
A compact real linear operator $R=C+A$ on $\h$ is in the Schatten $p$-class, denoted by $\B_p(\h)$, for $1\leq p <\infty$, if both its complex linear part $C$ and its antilinear part $A$ are in Schatten $p$-class. The Schatten $p$-norm of $R$ is defined as the sum
\begin{equation}
 \|R\|_p = \| C \|_p + \|A\|_p
\end{equation}
with 
$ \|C\|_p^p=\sum_{j=1}^{\infty} s_j(C)^p$  and $\|A\|_p^p=\sum_{j=1}^{\infty} s_j(A)^p$. 
Following standard nomenclature, operators in $\B_1(\h)$ are called trace-class (or nuclear), and those in $\B_2(\h)$ Hilbert-Schmidt. 
\end{defn}

This is just one way of defining the Schatten norm of a real linear operator. 
It is not clear how different approximations numbers in general are related to this definition of singular numbers for real linear operators; see \cite{rlinapprox} for real linear operators of finite rank. This choice of Schatten norm for a real linear operator is equivalent to its complexification in the following sense. 

\begin{prop}\label{oplusequiv}
Let $R=C+A$ be a compact real linear operator on $\h$ with complex linear part $C$ and antilinear part $A$, and let $R^\Complex$ be the complexification with respect to some unitary conjugation $\kappa$. Then $R^\Complex\in\B_p(\h\oplus\h)$ if and only if $R\in\B_p(\h)$. In addition, it holds  
$$
\|(C+A)^\Complex\|_p^p \geq \|C\|_p^p + \|A\|_p^p
$$
and 
$$
\| C^\Complex\|_p^p = 2 \|C\|_p^p, \quad \|A^\Complex\|_p^p = 2 \|A\|_p^p.
$$ 
\end{prop}
\begin{proof}
Assume $R^\Complex\in\B_p(\h\oplus\h)$. We have 
\[ 
\|R^\Complex\|^p_p= \sup \left\{ \sum_j | (R^\Complex h_j, h_j)_{\h\oplus\h} | \ :\  \{h_j\} \text{ orthonormal} \right\}.
\]
Choosing $\{h_j\}_{j=1}^\infty$ to be such that $h_{2j-1}= (e_j,0)$ and $ h_{2j} =(0,f_j)$, where $\{e_j\}$ and $\{f_j\}$ are orthonormal sets of $\h$, we get 
\[ 
 \sum_j |(R^\Complex h_j, h_j)_{\h\oplus\h}|^p = \sum_j |(Ce_j,e_j)|^p + \sum_j |(A\kappa f_j, f_j)|^p .
\]
Since the orthonormal sets $\{e_j\}$ and $\{f_j\}$ are arbitrary, we see that $C$ and $A\kappa$, hence and also $A$, are in $\B_p(\h)$. 

Assume $R=C+A \in\B_p(\h)$ so that by definition $C,A \in\B_p(\h)$. It follows for the complexifications $C^\Complex, A^\Complex\in\B_p(\h\oplus\h)$. Namely, as we have the representation $C=\sum_j s_j(C) (\cdot, e_j) f_j$ for some orthonormal bases $\{e_j\}$ and $\{f_j\}$, then $\kappa C \kappa = \sum_j s_j(C) (\cdot, \kappa e_j) \kappa f_j$. Then for all $(x,y)\in\h\oplus\h$ it holds 
\[
 \begin{bmatrix}
  C & \\ & \kappa C \kappa
 \end{bmatrix}
{\begin{bmatrix} x \\ y \end{bmatrix}}
= \sum_j s_j(C) \left( {\begin{bmatrix} x \\ y \end{bmatrix}} , {\begin{bmatrix} e_j \\ 0 \end{bmatrix}} \right) {\begin{bmatrix} f_j \\ 0 \end{bmatrix}} 
+ s_j(C) \left( {\begin{bmatrix} x \\ y \end{bmatrix}} , {\begin{bmatrix} 0 \\ \kappa e_j \end{bmatrix}} \right) {\begin{bmatrix} 0 \\ \kappa f_j \end{bmatrix}}
\]
where the sets $\{ h_j \}_{j=1}^{\infty}$ and $\{g_j\}_{j=1}^\infty$, with $h_{2j-1}=(e_j,0)$, $h_{2j} =(0,\kappa e_j)$ and $g_{2j-1}=(f_j,0)$, $g_{2j}=(0,\kappa f_j)$, form orthonormal bases of $\h\oplus\h$. This means that the singular values of $C^\Complex$ are those of $C$ but each repeated twice and therefore $\|C^\Complex\|_p^p = 2\|C\|_p^p$. 

In a similar fashion, the representations $A\kappa = \sum_j s_j(A) (\cdot, \kappa e_j) f_j$ and $\kappa A = \sum_j s_j(A) (\cdot, e_j) \kappa f_j$ lead to 
\[
 \begin{bmatrix}
   & A\kappa \\ \kappa A&  
 \end{bmatrix}
{\begin{bmatrix} x \\ y \end{bmatrix}}
= \sum_j s_j(A) \left( {\begin{bmatrix} x \\ y \end{bmatrix}} , {\begin{bmatrix} e_j \\ 0 \end{bmatrix}} \right) {\begin{bmatrix} 0 \\ \kappa f_j \end{bmatrix}} 
+ s_j(A) \left( {\begin{bmatrix} x \\ y \end{bmatrix}} , {\begin{bmatrix} 0 \\ \kappa e_j \end{bmatrix}} \right) {\begin{bmatrix} f_j \\ 0 \end{bmatrix}} .
\]
Again, it is seen that $\|A^\Complex\|_p^p = 2 \| A \|_p^p$. It follows $R^\Complex\in\B_p(\h\oplus\h)$. 
\end{proof}

Although the trace-class of real linear operators has been defined above, defining an actual trace for a real linear operator is seemingly more problematic. 
Namely, for an antilinear operator $A$, the sum $\sum_j (Ae_j, e_j)$ depends on the orthonormal basis $\{e_j\}$. 
Fixing an orthonormal basis of $\h$ allows  defining a basis-dependent trace of a real linear operator as the trace of its complexification with respect to this basis.

\section{Invariant subspaces}


As is well known, eigenspaces corresponding to specific eigenvalues determine invariant subspaces for a compact complex linear operator. Recall that a set $\M \subset\h$ is invariant under an operator $L$ when $L\M \subset \M$. Granted, a specific eigenvalue determines a real linear invariant subspace for a real linear operator. However, as there are real linear operators defined more naturally on complex linear Hilbert spaces, studying complex linear invariant subspaces is, for the same reason, more natural. This is also done to preserve the structure of the underlying complex Hilbert space. 

In general, for a real linear operator, having a complex linear invariant subspace is quite a restricting property. 

\begin{prop}
Let $\M$ be a complex linear invariant subspace under a real linear operator $R=C+A$ on $\h$, where $C$ and $A$ are the complex linear and antilinear parts of $R$. Then $\M$ is invariant under both $C$ and $A$. 
\end{prop}

\begin{proof}
Assume $x\in\M$. Then also $Cx =\frac 12(R-iRi)x \in \M$ and $Ax=\frac 12 (R+iRi) x \in\M$. 
\end{proof}

For this reason, we are concerned with compact antilinear operators and show that then there exists, much like in the complex linear case, a nontrivial closed complex linear invariant subspace. The investigation of the invariant subspace problem is an longstanding project. In 1930, von Neumann proved that every compact operator has a nontrivial closed invariant subspace. Aronszajn and Smith extended this to compact operators on Banach spaces \cite{as:invariant}. Using non-standard analysis, Bernstein and Robinson proved the version for polynomially compact operators on a Hilbert space \cite{br:invariant}, and Halmos proved the same using classical analysis \cite{halmos:invariant}. Lomonosov proved that there exists a nontrivial closed subspace of a Banach space that is invariant under every operator commuting with a given nonzero compact operator \cite{lomonosov:invariant}. Hilden simplified the proof of this \cite{michaels:hildensprf}. For recent advances and a modern approach on the invariant subspace problem in general, see \cite{cp:modappinvariant}. 

Let $p(z) = \sum_{j=0}^n a_j z^j$ be a polynomial with $a_j\in\Complex$. 
A complex polynomial in a real linear operator $R$ is defined to be the real linear operator $\sum_{j=0}^n a_jR^j$. 

\begin{defn}
For a compact antilinear operator $A$ on $\h$, define $\langle A \rangle$ to be the closure in the operator norm of polynomials in $A$ with complex coefficients.  
\end{defn}

Note that $\langle A \rangle$ is almost a subalgebra in the sense that if $\lambda\in\Complex$ and $T,S \in\langle A \rangle$, then $\lambda T$, $T+S$, $TS \in\langle A \rangle$. 
Obviously the main obstruction is $T(\lambda S) \neq \lambda TS$. In general, the elements of $\langle A \rangle$ are genuinely real linear operators since even powers of $A$ are complex linear while odd powers are antilinear. 
Also, $\langle A \rangle$ contains noncompact elements, e.g. all polynomials $p(A)=\sum_j a_j A^j \neq 0$ for which $a_0\neq 0$. 
Clearly $\M$ is a closed complex linear invariant subspace under $A$ if and only if it is such under every $T\in\langle A \rangle$.

Like in the complex linear case, invariant complex linear subspaces for an antilinear operator $A$ may be generated by forming the closure of the complex linear span of $\{A^jy\}_{j=0}^\infty$ for any nonzero $y\in\h$. Indeed, denoting by $\cspn_\Complex$ the closed complex linear span, 
\begin{equation}\label{cspn}
\cspn_\Complex \{A^jy\}_{j=0}^\infty = \langle A \rangle y = \{ Ty : T\in\langle A \rangle \},
\end{equation}
is a closed complex linear subspace of $\h$. Either $\langle A \rangle y$ is a closed nontrivial invariant subspace under $A$, or $\langle A \rangle y=\h$ whence $y$ is called a cyclic vector of $A$.

The arguments of Lomonosov and Hilden \cite{michaels:hildensprf} largely apply in proving the following. Of course, antilinearity requires some changes. 
Indeed, the fact that \eqref{cspn} is a complex linear subspace of $\h$ is crucial. 

\begin{thm}\label{thmlomo}
Let $A$ be a nonzero compact antilinear operator on $\h$. Then there exists a nontrivial closed complex linear subspace of $\h$ that is invariant under every operator $T\in\langle A \rangle$.
\end{thm}

\begin{proof}
Assume the contrary, that there is no such subspace. If $\lambda\in\Complex$ is an eigenvalue of $A$ with the corresponding eigenvector $x\in\h$, then $\spn_\Complex \{x\}$ provides a contradiction. Thus it must be that $\sigma(A)=\{0\}$, see \cite[Theorem 2.21]{reallinearoperator}. Without loss of generality, we may assume $\|A\|=1$. 

Let $x_0\in\h$ be such that $\|Ax_0\|>1$. Then we have 
\begin{equation}\label{lomox0ineq}
\|x_0\| = \|A\|\|x_0\| \geq \|Ax_0\| > 1
\end{equation}
from which it follows that $0\notin \overline B_1(x_0)$, the closed unit ball with centre at $x_0$.  Moreover, we have $\overline{A(\overline B_1(x_0))} \subset \overline B_1(Ax_0)$, the closed unit ball with centre at $Ax_0$, because for all $x\in\overline B_1(x_0)$ it holds $\|Ax-Ax_0\| \leq \|A\| \|x-x_0\| \leq 1$. Thus, $0\notin\overline{A(\overline B_1(x_0))}$ since $\|Ax_0\|>1$. 

For all nonzero $y\in\h$, 
$$ \V_y  = \{ p(A) y : p(A) \text{ is a complex polynomial in }A \}$$ 
is a complex linear subspace invariant under every polynomial $p(A) \in\langle A \rangle$. Thus its closure is a complex linear closed subspace invariant under every $T\in\langle A \rangle$. As $\V_y\neq \{0\}$ it is dense in $\h$ due to the counter-assumption. It follows that for every nonzero $y\in\h$ there exists a polynomial $p(A)\in\langle A \rangle$ such that $\|p(A)y-x_0\|<1$. For each polynomial $p(A)\in\langle A \rangle$, define the open set $U_{p(A)}=\{ y\in\h : \|p(A)y-x_0\|<1\}$. The collection $\{ U_{p(A)} : p(A)\text{ is a polynomial in } A\}$ is an open cover of $\h\setminus\{0\}$ and thus of $\overline{A(\overline B_1(x_0))}$ since $0\notin \overline{A(\overline B_1(x_0))}$. By the compactness of $A$, $\overline{A(\overline B_1(x_0))}$ is compact, and there is a finite subset $\p_0 \subset \langle A\rangle$ of polynomials such that $\{U_{p(A)} : p(A) \in\p_0\}$ covers $\overline{A(\overline B_1(x_0))}$. Consequently, for every $y\in\overline{A(\overline B_1(x_0))}$, there exists a $p(A)\in\p_0$ such that $\|p(A)y-x_0\|<1$. 

Since $Ax_0\in A\overline B_1(x_0)$, there exists $p_1(A) \in \p_0$ such that $\|p_1(A) A x_0 - x_0\| < 1$. 
Thus $ A p_1(A) A x_0 \in A\overline B_1(x_0)$, and there exists $p_2(A) \in \p_0$ such that $\| p_2(A) A p_1(A) A x_0-x_0\|<1$. 
By induction, we can define a sequence $\{p_n(A)\}$ in $\p_0$ such that 
\begin{equation}\label{lomoseq}
\|p_n(A) A \cdots p_2(A) A p_1(A) A x_0-x_0\| < 1 .
\end{equation} 

For a polynomial $p(A)=\sum_{j=0}^N a_j A^j$, the polynomial with conjugated coefficients $\overline p(A)=\sum_{j=0}^N \overline{a_j} A^j$ is the unique polynomial that satisfies $A p(A)=\overline p(A) A$. Suppose there are two polynomials $p_1(A)$ and $p_2(A)$ for which $Ap(A) = p_1(A) A=p_2(A) A$. Then $q(A)=0$, where $q(A)=p_1(A)A-p_2(A)A = \sum_{j=1}^{N+1} b_j A^j$. But the coefficients of $q(A)$ must be all zero, for if the leading coefficient $b_k\neq 0$, $1\leq k \leq N+1$, then $\spn_\Complex \{ A^j z \}_{j=0}^{k-1} $ with an arbitrary $z\neq 0$ is a complex nontrivial closed subspace under every $T \in \langle A \rangle$. This is in contradiction with the counter-assumption. 

Let $M=\max \{ \|p(A)\|, \| \overline p (A) \| : p(A) \in \p_0 \}$. Then 
\begin{equation}\label{lomolessthan} 
\| p_n(A) A \cdots p_2(A) A p_1(A) A \| = \| \tilde p_n(A) \cdots \tilde p_1(A) A^n \| \leq \| (MA)^n\| , 
\end{equation}
for all $n=1,2,\ldots$, where $\tilde p_k(A)$ is either $p_k(A)$ or $\overline p_k(A)$. 

\begin{lemma}\label{lemnolla}
It holds that $\lim_{n\to\infty} \| (MA)^n \|=0$. 
\end{lemma}
\begin{proof}
The complex linear compact operator $(MA)^2$ has no eigenvalues, and the orgin is the only spectral value. For if it had an eigenvalue $\lambda$ with the eigenvector $x$, then $\spn_\Complex \{x,Ax\}$ would be a nontrivial complex linear invariant subspace, contrary to the counter-assumption. Thus, by the spectral radius formula, it holds that 
$ \lim_{n\to\infty} \|(MA)^{2n}\|^{1/n} = 0$ which implies in particular that $\lim_{n\to\infty} \| (MA)^{2n} \|=0$. Then also $\lim_{n\to\infty} \|(MA)^{2n+1}\| =0$ which proves the claim. 
\end{proof}

Using \eqref{lomolessthan} with Lemma \ref{lemnolla}, we get $\lim_{n\to\infty} \| p_n(A) A \cdots p_1(A) A \| = 0$ which combined to \eqref{lomoseq} leads to $\|x_0\| \leq 1$, a contradiction to \eqref{lomox0ineq}, thus proving the antilinear Lomonosov's theorem. 
\end{proof}

It is not clear how Theorem \ref{thmlomo} could be strengthened. 
In the proof above it is crucial that $\langle A \rangle y$ yields a complex linear subspace. 
In general, for a compact real linear  operator $R$ a similarly defined set $\langle R \rangle y$ cannot be expected to be invariant under $R$. 
In particular, a compact real linear operator may not have any nontrivial complex linear invariant subspaces. 
An example of this is the operator
\begin{equation*}
 \begin{bmatrix} 1 & 1 \\ 0 & 1 \end{bmatrix} + \begin{bmatrix} 0 & 0 \\ 1 & 0 \end{bmatrix} \tau
\end{equation*}
on $\Complex^2$, where $\tau$ denotes the componentwise complex conjugation. 

Finally, let us comment on the connection between eigenvalues and finite dimensional invariant subspaces. 
In the complex linear case, the existence of a finite dimensional invariant subspace entails the existence of eigenvalues. 
This reasoning is not valid for real linear operators. 
An obstruction is given by the following proposition, where the minimum modulus (or injectivity modulus) is defined as $j(R) = \inf\{ \| Rx\| : \|x\| = 1 \}$.  

\begin{prop}[\cite{huhtanen:rlineigprob}] \label{noeigval}
Let a real linear operator $R=C+A$ be split as $R=\hat R + T$, where $T=\frac 12 (A-A^*)$ is the skew-adjoint part of the antilinear part of $R$ and $\hat R = C + \frac 12 (A+A^*)$. If $\|\hat R\| < j(T)$, then $R$ has no eigenvalues. 
\end{prop}
\begin{proof}
By the antilinear skew-adjointness of $T$, $Tx$ is always orthogonal to $x$ for any $x\in\h$. Assume that $\lambda$ is an eigenvalue of $R$ with eigenvector $x\in\h$ of unit length. Then we have by the Pythagorean theorem
$$ |\lambda|^2 + \| Tx\|^2 = \| \lambda x - Tx \|^2 = \| \hat R x\|^2 .
$$
Taking infimum on the left-hand side and supremum on the right-hand side over unit vectors, we have
$$ | \lambda |^2 + j(T)^2 \leq \| \hat R \|^2
$$
which contradicts $\| \hat R \| < j(T)$. 
\end{proof}

When $\h$ is infinite dimensional, the origin is always in the spectrum of a compact real linear operator. 
However, all points in the spectrum of a compact real linear operator outside the origin are eigenvalues. 
In this regard, to determine whether there is an associated invariant subspace Proposition \ref{noeigval} is crucial. 
Indeed, assume the Hilbert space splits into the direct sum $\h = \bigoplus_j \h_j$, where each $\h_j$ is a finite dimensional invariant subspace for the compact real linear operator $R$. 
If the assumption $\| \hat R_j \| < j(T_j)$ holds for the restrictions $R_j$ of $R$ to $\h_j$, then $R$ has no eigenvalues. 
In particular, from this it can be seen that $\sigma(R) = \{0\}$ does not imply $\lim_{n\to\infty} \| R^n\|^{1/n} = 0$ for real linear operators in general. 

For this reason, the eigenvalue problem for compact real linear  operators is challenging in general. 
Also, it underlines the spectral questions for finite rank real linear operators, which we will discuss next.

\section{Characteristic polynomials and numerical functions}

The real linear eigenvalue problem, i.e.,  characterizing the points $\lambda\in\Complex$ for which  $R-\lambda$ is singular, is a nontrivial task already in finite dimensions. 
We treat the finite dimensional case first in terms of the characteristic polynomial of $R$ and a so-called numerical function. 
Then we discuss a slightly different numerical function for real linear trace-class operators which is defined through a limit process of finite rank operators. 

Finite rank operators are best dealt with real linear matrix analysis. I.e., we assume having a represention of $R$ on $\Complex^n$ as 
\begin{equation}
 R = C+ B\tau
\end{equation}
with $C,B\in\Complex^{n\times n}$. Here $\tau$ is the componentwise complex conjugation $\tau z = \overline z$ for $z\in\Complex^n$. This representation is attained with respect to some orthonormal basis $\{e_j\}_{j=1}^n$ of the complex linear span of $R\h$.  

\begin{defn}
The characteristic polynomial of a real linear operator $R$ on $\Complex^n$ is defined as the polynomial 
\begin{equation}
p (\lambda , \overline \lambda ) = \det(R-\lambda)^\Complex 
= \det \begin{bmatrix} C -\lambda & B \\ \overline B & \overline{C-\lambda} \end{bmatrix}. 
\end{equation}
\end{defn}

The characteristic polynomial is not a complex analytic polynomial but a so-called real analytic polynomial. Calling it characteristic is justified by the fact that 
\begin{equation}
 \sigma (R) = \{ \lambda \in \Complex : p(\lambda, \overline\lambda ) =0 \}. 
\end{equation}
Polynomials defined in this way have been studied in \cite{eirolaetal:solmeth} and \cite{rlinmatan}, where they are called characteristic bivariate polynomials resulting from the separation of $\lambda=\alpha+i\beta$ into its real and imaginary parts. 

The characteristic polynomial can be cast in a canonical form involving complex analytic polynomials. 

\begin{thm}\label{basicsofcp}
The characteristic polynomial $p (\lambda, \overline\lambda)=\sum_{i,j=0}^{n} h_{ij} \lambda^j \overline\lambda^i$ of a real linear operator $R$ on $\Complex^n$ satisfies the following: 
\begin{enumerate}[(i)]
\item 
$p$ is real valued. 
\item 
The coefficient matrix $H=[h_{ij}]_{i,j=0}^n$ is Hermitian. 
\item 
Let $U=[u_{ij}]_{i,j=0}^{n}\in\Complex^{(n+1)\times (n+1)}$ be unitary such that $UHU^*$ is diagonal. Then $p$ can be cast in the form
\begin{equation}\label{canonform}
p (\lambda, \overline\lambda ) = \sum_{i=0}^{n} d_i |p_i(\lambda)|^2 ,
\end{equation}
where $d_i\in\Real$ and 
$$
p_i(\lambda) = \sum_{j=0}^n u_{ij} \lambda^j \qquad \text{for } i=0,1,\ldots,n .
$$ 
The polynomials $p_i$ have no common zeros. 
\item 
If $H$ is positive definite, then $p $ can be cast in the form 
\[ 
p ( \lambda , \overline\lambda ) = \sum_{i=0}^n | p_i(\lambda) |^2 , 
\]
where each $p_i$ is a complex polynomial of degree $i$. The polynomials $p_i$ have no common zeros, and consequently, $\sigma(R) = \emptyset$.   
\item 
If $\det R^\Complex \leq 0$, then $p (\lambda, \overline\lambda)$ has a zero, and consequently, $\sigma(R) \neq\emptyset$.  
\item 
For the adjoint $R^*$, it holds $p_{R^*}(\lambda, \overline\lambda) = p (\overline\lambda, \lambda )$, i.e., the coefficients of $p_{R^*}$ are complex conjugates of $p $. In particular, the characteristic polynomial of a self-adjoint operator has real coefficients. 
\end{enumerate}
\end{thm}

\begin{proof}
By properties of the determinant, the characteristic polynomial is of degree $2n$ and can be represented as 
\begin{equation}
p (\lambda, \overline\lambda) = \sum_{i,j=0}^n h_{ij} \lambda^j \overline\lambda^i = v_\lambda^* H v_\lambda, 
\end{equation}
where $v_\lambda = (1,\lambda,\lambda^2,\ldots,\lambda^n) \in \Complex^{n+1}$ and  $H=[h_{ij}]_{i,j=0}^n\in\Complex^{(n+1)\times (n+1)}$ is the coefficient matrix. 
\begin{enumerate}[(i)]
\item 
Since 
\begin{align*}
p ( \lambda, \overline \lambda) 
&= 
\det \begin{bmatrix} C -\lambda & B \\ \overline B & \overline{C-\lambda} \end{bmatrix} 
\\
&
= 
\det \begin{bmatrix} 0 & I \\ I & 0 \end{bmatrix} 
\begin{bmatrix} C -\lambda & B \\ \overline B & \overline{C-\lambda} \end{bmatrix} 
\begin{bmatrix} 0 & I \\ I & 0 \end{bmatrix}
= 
\overline{ p (\lambda, \overline\lambda) } ,
\end{align*}
$p $ is real valued.

\item 
Since $p (\lambda, \overline \lambda)$ is real valued, the two polynomials 
$$
p (\lambda, \overline \lambda) = \sum_{i,j=0}^n h_{ij} \lambda^j \overline \lambda^i
$$ 
and  
$$
\overline{p (\lambda, \overline \lambda )} = \sum_{i,j=0}^n \overline{h_{ij}} \overline \lambda^j \lambda^i=\sum_{i,j=0}^n \overline{h_{ji}} \lambda^j \overline \lambda^i
$$ 
are identical for all $\lambda\in\Complex$. Then the coefficients corresponding to identical powers $\lambda^j\overline\lambda^i$ must also be identical, i.e., $h_{ij} = \overline{h_{ji}}$. 

\item 
As the coefficient matrix is Hermitian, we have $H = U^*DU$ with $U=[u_{ij}]$ unitary and $D$ diagonal with real diagonal entries $d_i$. Then 
$$ 
p (\lambda, \overline \lambda) = v_\lambda^* H v_\lambda = (Uv_\lambda)^* D U v_\lambda = \sum_{i=0}^n d_i|p_i(\lambda)|^2 ,
$$
where $p_i(\lambda) = \sum_{j=0}^n u_{ij} \lambda^j$. The polynomials cannot have common zeros since, if there exists  $\mu\in\Complex$ such that $p_i(\mu)=0$ for all $i=0,\ldots,n$, then $U v_\mu = 0$ contradicting the unitarity of $U$.  

\item 
Let $P$ be the order reversing permutation matrix $(x_0,x_1,\ldots,x_n) \mapsto (x_n, \ldots, x_1, x_0)$. Then $p (\lambda, \overline \lambda) = u_\lambda^* P^* H P u_\lambda$, where $u_\lambda = P v_\lambda = (\lambda^n, \cdots, \lambda, 1)$. The matrix $\tilde H = P^* H P$ is  Hermitian and positive definite so that it has a Cholesky decomposition as $\tilde H = T^*T$, where $T$ is upper triangular with strictly positive diagonal elements. Then 
$$ p (\lambda, \overline \lambda) = u_\lambda^* T^* T u_\lambda  = (T u_\lambda)^* T u_\lambda = \sum_{i=0}^n |p_i(\lambda)|^2 ,$$
where $p_i$ is of degree $i$. The existence of a zero common to all of the polynomials $p_i$ would contradict the invertibility of $T$.  Then $p(\lambda, \overline\lambda )>0$ for all $\lambda\in\Complex$, or equivalently $\sigma(R) = \emptyset$. 

\item 
The leading term of $p(\lambda, \overline \lambda)$ is $|\lambda|^{2n}$. Restricted to the real line $f(r)=p(r,r)$ is a continuous function with $f(0)=\det R^{\Complex} \leq 0$ and $f(r) \to +\infty$ as $r\to\infty$ which proves the claim. 

\item 
By the fact that taking the transpose leaves the determinant unchanged and the characteristic polynomials are real valued, it holds 
\begin{align*}
 p (\overline\lambda, \lambda) 
&= \det
\begin{bmatrix}
 C - \overline\lambda & B \\
\overline B & \overline C - \lambda
\end{bmatrix}
= \det \begin{bmatrix}
   C^* - \lambda & B^T \\
  B^* & \overline{C^*} - \overline\lambda
  \end{bmatrix}
\\& =p_{R^*}(\lambda, \overline\lambda). 
\end{align*}
Denoting the coefficient matrix of $p_{R^*}$ by $[r_{ij}]$, it follows that 
$$ 
p_{R^*} (\lambda, \overline\lambda) 
= \sum_{i,j=0}^n r_{ij} \lambda^j \overline\lambda^i 
= \overline{p (\overline\lambda , \lambda ) } 
= \sum_{i,j=0}^n \overline{h_{ij}} \lambda^j \overline\lambda^i
$$ 
for all $\lambda\in\Complex$. Thus $r_{ij} = \overline{h_{ij}}$. Evidently then, for self-adjoint $R$, the coefficients are real. \qedhere
\end{enumerate}
\end{proof}

Theorem \ref{basicsofcp} above connects the characteristic polynomials of real linear operators to the study of positive polynomials and sums of squares which is classical. 
Sum of squares decompositions arise in classical moment problems, real algebraic geometry, optimization, and control theory problems. 
See \cite{heltonputinar:positive} for a thorough exposition. 

The coefficient matrix $H$ cannot be negative definite since the leading coefficient is $h_{nn}=e_n^*He_n=1$, where $e_n=(0,\ldots,0,1)$. 
Indefiniteness of $H$ is a necessary condition for the nonemptiness of $\sigma(R)\setminus\{0\}$ but not sufficient. 
An example of such an operator is 
\begin{equation}\label{countertofov}
 \begin{bmatrix}
  0 & A \\ A & 0
 \end{bmatrix}
\quad\text{with}\quad 
A = \begin{bmatrix} \sqrt{\frac{1+\eps}{2}} & \sqrt{\frac{1-\eps}{2}} \\ -\sqrt{\frac{1-\eps}{2}} & \sqrt{\frac{1+\eps}{2}}
    \end{bmatrix} 
\text{ and } 0<\eps<1.
\end{equation}
The characteristic polynomial is then $p(\lambda,\overline\lambda) = |\lambda|^4 -2\eps|\lambda|^2 + 1$ for which the coefficient matrix is diagonal with entries $1,-2\eps,1$. 

Recall that the field of values (or synonymously numerical range) of a matrix $A\in\Complex^{n\times n}$ is the closed convex subset of the complex plane
$$
 W(A) = \{ x^*Ax : \|x\|=1 \}. 
$$
For the field of values, see, e.g. \cite{hornjohnson1}. 
By being the closed interval between the smallest and largest eigenvalue of $H$, the field of values of a Hermitian matrix $H$ is easily computed.


In terms of the field of values, Theorem \ref{basicsofcp} clarifies that, strictly speaking, any given Hermitian matrix does not give rise to the characteristic polynomial of a finite rank real linear operator. 
Furthermore, on $\Complex$ the real linear operator $\alpha+\beta\tau$ with $\alpha,\beta\in\Complex$ has the characteristic polynomial $p(\lambda, \overline\lambda) = |\alpha|^2-|\beta|^2 - \overline\alpha \lambda - \alpha\overline\lambda + \lambda\overline\lambda$. Thus, there are Hermitian matrices that are not scalable or translatable to the coefficient matrix of a characteristic polynomial. 

To connect further the field of values of the coefficient matrix $H$ to the characteristic polynomial $p$ of a real linear operator $R$, the following normalization is introduced. 

\begin{defn}
Let $R$ be a real linear operator on $\Complex^n$ with the characteristic polynomial $p$. 
Then 
$$
F  :\Complex \to \Real, \quad F (\lambda) = \frac{ p (\lambda, \overline\lambda) }{ \sum_{j=0}^n |\lambda|^{2j} } = \frac{ v_\lambda^* H v_\lambda }{ \|v_\lambda\|^2}  
$$
is said to be the numerical function of $R$, where $v_\lambda=(1,\lambda,\ldots,\lambda^n)$. 
\end{defn}

Based on the properties of the characteristic polynomial, we have the following. 

\begin{thm}\label{specialF}
Let $F  :\Complex \to\Real$ be the numerical function of a real linear operator $R$ on $\Complex^n$, and let the coefficient matrix of the characteristic polynomial be diagonalized as $H=U^*DU$ with $D$ diagonal and $U$ unitary. Then the following hold:
\begin{enumerate}[(i)]
\item 
$F $ is continuous and its range is a subset of $W(H)$. 
\item 
$F (\lambda)=0$ if and only if $p (\lambda, \overline\lambda)=0$. 
\item 
For every $\lambda\in\Complex$, $F(\lambda)$ is a convex combination of eigenvalues of $H$,
\begin{equation}\label{convcomb}
F(\lambda) = \sum_{i=0}^n d_i \frac{ |p_i(\lambda)|^2 }{ \sum_{j=0}^n |p_j(\lambda)|^2 }, 
\end{equation}
where the polynomials $p_i$ and the real numbers $d_i$ are as in Theorem \ref{basicsofcp} (iii). 
\item 
In particular, 
\begin{align*}
F(0) 
&= \det R^\Complex = \sum_{i=0}^n d_i |u_{i0}|^2, 
\\
\lim_{|\lambda|\to\infty} F(\lambda) 
&=  \sum_{i=0}^n d_i |u_{in}|^2 = 1 .
 \end{align*}
so that
$$
[ \det R^\Complex , 1 )   \subset F  (\Complex) \subset W(H) = [\lambda_{\min} (H), \lambda_{\max} (H) ], 
$$
where 
$\lambda_{\min}(H)$ and $\lambda_{\max}(H)$ are the smallest and greatest eigenvalue of $H$. 
\end{enumerate}
\end{thm}

\begin{proof}
\begin{enumerate}[(i)]
\item 
From the form of $F $ it is seen that $F $ is continuous from $\Complex$ to $\Real$. For any $\lambda\in\Complex$ 
$$ 
F(\lambda) = \left(\frac{ v_\lambda }{ \|v_\lambda\| } \right)^* H \left(\frac{ v_\lambda }{ \|v_\lambda\| } \right) \in W(H).
$$
\item 
As $ (\sum_{j=0}^n |\lambda|^{2j})^{-1} > 0$ for all $\lambda\in\Complex$, $p $ vanishes exactly when $F $ vanishes. 
\item 
Let $p (\lambda, \overline\lambda)= \sum_{i=0}^{n} d_i |p_i(\lambda)|^2 $ be the sum representation of Theorem \ref{basicsofcp} (iii). Since $U$ is unitary, there holds $\sum_{i=0}^n |\lambda|^{2j} = \sum_{i=0}^n |p_i(\lambda)|^2$. Thus 
\begin{equation*}
F(\lambda) = \sum_{i=0}^n d_i \frac{ |p_i(\lambda)|^2 }{ \sum_{j=0}^n |p_j(\lambda)|^2 }, 
\end{equation*}
which is a convex combination of eigenvalues $\{d_i\}$ of $H$. 
\item From the definition of $F $ it is seen that $F(0) = \det R^\Complex$ and from Equation \eqref{convcomb} that 
$$ F(0) = \sum_{i=0}^n d_i |u_{i0}|^2 .
$$
It holds 
$$
\lim_{|\lambda|\to\infty} F(\lambda) = \frac{ p (\lambda, \overline\lambda) }{ \sum_{j=0}^n|\lambda|^{2j} } = \frac{|\lambda|^{2n} + (\text{lower order terms}) }{ |\lambda|^{2n} + (\text{lower order terms}) } = 1.
$$
On the other hand $\frac{ |p_i(\lambda)|^2 }{ \sum_{j=0}^n |\lambda|^{2j} } \to  |u_{in}|^2$ as $|\lambda| \to \infty$ so that
$$
\lim_{|\lambda|\to\infty} F(\lambda) = \sum_{i=0}^n d_i |u_{in}|^2 ,
$$ 
which proves the equality. 

As $F $ is continuous and $H$ is Hermitian, we have the following inclusions: 
\begin{equation*}
[ \det R^\Complex , 1 )   \subset F  (\Complex) \subset W(H) = [\lambda_{\min} (H), \lambda_{\max} ( H)] . \qedhere
\end{equation*}
\end{enumerate}
\end{proof}

The numerical function $F$ can be defined for an arbitrary Hermitian matrix $H$ or even for any matrix $A$. 
As exemplified by \eqref{countertofov}, the closure of $F(H)$ need not equal $W(H)$. 
It is of independent interest to inquire how much of the field of values $W(H)$ remains uncovered by $F(\Complex)$ in each case. 
Numerical evidence suggests that some always remain. 
The interest stems from the fact that calculating the range of $F(\Complex )$ is computationally a one dimensional problem whereas the field of values is defined on the unit sphere of $\Complex^n$.

In the subsequent defintion of the numerical function for trace-class operators, we will need the following rotation property. 

\begin{prop}
Let $F$, $F_{-\theta}$ be the numerical functions and $p$, $p_{-\theta}$ the characteristic polynomials of finite rank real linear operators  $C+B\tau$ and $e^{-i\theta} C + B\tau$, where $\theta\in\Real$. Then it holds for all $r\geq 0$
\begin{equation*}
F (e^{i\theta}r )= F_{-\theta} (r) \quad\text{and}\quad p ( e^{i\theta} r , e^{-i\theta} r ) = p_{-\theta} ( r , r ) . 
\end{equation*}
\end{prop}
\begin{proof}
By properties of the determinant,
\begin{align*}
p_{-\theta} (r,r)
&= \det \begin{bmatrix} e^{i\theta/2}I \\  & e^{-i\theta/2} I \end{bmatrix} 
 \begin{bmatrix} e^{-i\theta}C - r & B \\ \overline B & e^{i\theta} \overline C -r \end{bmatrix}  
\begin{bmatrix} e^{i\theta/2}I \\  & e^{-i\theta/2} I \end{bmatrix}
\\&
= p (e^{i\theta} r , e^{-i\theta} r ) . 
\end{align*}
As $v_\lambda^*v_\lambda$ with $\lambda = e^{i\theta} r$ is indepenent of the argument $e^{i\theta}$, the identity holds for the numerical function as well. 
\end{proof}


Recall that for a complex linear trace-class operator $C$ the determinant can be extended in the following way
\begin{equation*}
\det (I+C) = \lim_{n\to\infty} \det (I+C_n) ,
\end{equation*}
where $\|C_n-C\|_1\to 0$ as $n\to\infty$. Furthermore, the function 
$
f : \Complex \to \Complex$,   $f(z) =  \det (I + zC)
$ 
is entire and depends continuously on $C$ in the sense 
$$ 
| \det(I + C_1) - \det (I + C_2) | \leq \| C_1-C_2\|_1 \exp (1+ \|C_1\|_1 + \|C_2\|_1 ), 
$$
where $C_1$ and $C_2$ are complex linear trace-class operators. Also, $I+C$ is boundedly invertible if and only if $\det(I+C) \neq 0$. See \cite[Section XI.9]{DS2} or  \cite[Chapter II]{GGK:det} for details. 
This allows to define the following. 

\begin{defn}
Let $R=C+A$ be a real linear trace-class operator on $\h$. Define its characteristic function for $\lambda = r e^{i\theta} \neq 0$ as 
$$ \varphi (r e^{i\theta}) = \det \left[ I - r^{-1} (e^{-i\theta} C + A)  \right]^\Complex .
$$
\end{defn}

\begin{prop}\label{limitbehav}
The characteristic function $\varphi $ of a real linear trace-class operator on $\h$ is continuous on $\Complex\setminus\{0\}$ and its zeros are the eigenvalues of $R$. 
Furthermore, for all $\lambda\neq 0$
$$ 
 \varphi (\lambda)  = \lim_{n\to\infty} \frac{p_{n}(\lambda, \overline \lambda)}{ |\lambda|^{2n} } ,  
$$
where $R_n=C_n+A_n$ are finite rank operators with characteristic polynomials $p_n(\lambda, \overline\lambda )$ such that $\|R_n-R\|_1\to 0$. 
\end{prop}

\begin{proof}
For fixed $\theta\in\Real$, we have
$ \varphi (re^{i\theta}) =  f_{\theta} (\frac 1r)$, where $f_{\theta} (z) = \det (I - z(e^{-i\theta}C+A)^\Complex )$ is an entire function. Thus restricted to the half-line $\{ re^{i\theta} : r>0\}$,  $\varphi $ is continuous. Furthermore, by the properties of $f(z)$,  $f(r^{-1})=0$ with a positive real $r^{-1}$ if and only if  $r$ is an eigenvalue of $(e^{-i\theta}C+A)^\Complex$. It is known \cite{reallinearoperator} that $\Real \cap \sigma(e^{-i\theta}C+A) = \Real \cap \sigma((e^{-i\theta}C+A)^\Complex)$. 

On the other hand, for fixed $r>0$, we have that 
\begin{align*}
& | \varphi (re^{i\theta}) - \varphi (re^{i\psi})| \\
& \leq |e^{i\theta}-e^{i\psi}| \| C\|_1 \exp (1+ \|(e^{-i\theta}C+A)^\Complex\|_1 + \| (e^{-i\psi} C + A)^\Complex \|_1 ). 
 \end{align*}
Thus $\varphi$ is continuous in the angular direction as well, and continuous in $\Complex\setminus\{0\}$. 

For the limit, we have for all $\lambda=r e^{i\theta} \in\Complex\setminus\{0\}$ 
\begin{align*}
 \frac{p_{R_n}(r e^{i\theta}, r e^{-i\theta})}{r^{2n}} 
=&  \frac{ p_{e^{-i\theta}C_n + A_n}(r,r)}{r^{2n}} \\
=& 
 \det (I - r^{-1} (e^{-i\theta}C_n+A_n))^{\Complex} 
\to \varphi(re^{i\theta}) 
\end{align*}
as $n\to\infty$ by properties of the $2n$-by-$2n$ determinant defining the characteristic polynomial.
\end{proof}

In Proposition \ref{limitbehav} use has been made of Proposition \ref{oplusequiv} in the equivalence of convergence of real linear operators and their complexifications in Schatten norms.



\bibliographystyle{plain}
\bibliography{references_santtu}

\def\cprime{$'$} \def\cprime{$'$}
\begin{thebibliography}{10}

\bibitem{as:invariant}
N.~Aronszajn and K.~T. Smith.
\newblock Invariant subspaces of completely continuous operators.
\newblock {\em Ann. of Math. (2)}, 60:345--350, 1954.

\bibitem{astalapaivarinta}
Kari Astala and Lassi P{\"a}iv{\"a}rinta.
\newblock Calder\'on's inverse conductivity problem in the plane.
\newblock {\em Ann. of Math. (2)}, 163(1):265--299, 2006.

\bibitem{br:invariant}
Allen~R. Bernstein and Abraham Robinson.
\newblock Solution of an invariant subspace problem of {K}. {T}. {S}mith and
  {P}. {R}. {H}almos.
\newblock {\em Pacific J. Math.}, 16:421--431, 1966.

\bibitem{cp:modappinvariant}
Isabelle Chalendar and Jonathan~R. Partington.
\newblock {\em Modern approaches to the invariant-subspace problem}, volume 188
  of {\em Cambridge Tracts in Mathematics}.
\newblock Cambridge University Press, Cambridge, 2011.

\bibitem{DS2}
Nelson Dunford and Jacob~T. Schwartz.
\newblock {\em Linear operators. {P}art {II}}.
\newblock Wiley Classics Library. John Wiley \& Sons Inc., New York, 1988.
\newblock Spectral theory. Selfadjoint operators in Hilbert space, With the
  assistance of William G. Bade and Robert G. Bartle, Reprint of the 1963
  original, A Wiley-Interscience Publication.

\bibitem{eirolaetal:solmeth}
Timo Eirola, Marko Huhtanen, and Jan von Pfaler.
\newblock Solution methods for {$\mathbb R$}-linear problems in {$\mathbb
  C^n$}.
\newblock {\em SIAM J. Matrix Anal. Appl.}, 25(3):804--828 (electronic), 2003.

\bibitem{GGK:det}
Israel Gohberg, Seymour Goldberg, and Nahum Krupnik.
\newblock {\em Traces and determinants of linear operators}, volume 116 of {\em
  Operator Theory: Advances and Applications}.
\newblock Birkh\"auser Verlag, Basel, 2000.

\bibitem{halmos:invariant}
P.~R. Halmos.
\newblock Invariant subspaces of polynomially compact operators.
\newblock {\em Pacific J. Math.}, 16:433--437, 1966.

\bibitem{heltonputinar:positive}
J.~William Helton and Mihai Putinar.
\newblock Positive polynomials in scalar and matrix variables, the spectral
  theorem, and optimization.
\newblock In {\em Operator theory, structured matrices, and dilations},
  volume~7 of {\em Theta Ser. Adv. Math.}, pages 229--306. Theta, Bucharest,
  2007.

\bibitem{hornjohnson1}
Roger~A. Horn and Charles~R. Johnson.
\newblock {\em Matrix analysis}.
\newblock Cambridge University Press, Cambridge, 1990.
\newblock Corrected reprint of the 1985 original.

\bibitem{rlinapprox}
Marko Huhtanen and Olavi Nevanlinna.
\newblock Approximating real linear operators.
\newblock {\em Studia Math.}, 179(1):7--25, 2007.

\bibitem{rlinmatan}
Marko Huhtanen and Olavi Nevanlinna.
\newblock Real linear matrix analysis.
\newblock In {\em Perspectives in operator theory}, volume~75 of {\em Banach
  Center Publ.}, pages 171--189. Polish Acad. Sci., Warsaw, 2007.

\bibitem{reallinearoperator}
Marko Huhtanen and Santtu Ruotsalainen.
\newblock Real linear operator theory and its applications.
\newblock {\em Integral Equations Operator Theory}, 69(1):113--132, 2011.

\bibitem{huhtanen:rlineigprob}
Marko Huhtanen and Jan von Pfaler.
\newblock The real linear eigenvalue problem in {$\mathbb C^n$}.
\newblock {\em Linear Algebra Appl.}, 394:169--199, 2005.

\bibitem{lomonosov:invariant}
V.~I. Lomonosov.
\newblock Invariant subspaces of the family of operators that commute with a
  completely continuous operator.
\newblock {\em Funkcional. Anal. i Prilo\v zen.}, 7(3):55--56, 1973.

\bibitem{michaels:hildensprf}
A.~J. Michaels.
\newblock Hilden's simple proof of {L}omonosov's invariant subspace theorem.
\newblock {\em Adv. Math.}, 25(1):56--58, 1977.

\bibitem{putinarshapiro}
Mihai Putinar and Harold~S. Shapiro.
\newblock The {F}riedrichs operator of a planar domain.
\newblock In {\em Complex analysis, operators, and related topics}, volume 113
  of {\em Oper. Theory Adv. Appl.}, pages 303--330. Birkh\"auser, Basel, 2000.

\bibitem{putinarshapiro2}
Mihai Putinar and Harold~S. Shapiro.
\newblock The {F}riedrichs operator of a planar domain. {II}.
\newblock In {\em Recent advances in operator theory and related topics
  ({S}zeged, 1999)}, volume 127 of {\em Oper. Theory Adv. Appl.}, pages
  519--551. Birkh\"auser, Basel, 2001.

\end{thebibliography}
\end{document}